\documentclass{amsart}

\usepackage[english]{babel}
\usepackage[utf8x]{inputenc}
\usepackage{amsmath}
\usepackage{graphicx}
\usepackage{enumerate}
\usepackage{amsmath}
\usepackage{mathrsfs}
\usepackage{mathtools}
\usepackage[backend=biber,url=true]{biblatex}
\renewbibmacro{in:}{%
  \ifentrytype{article}{}{\printtext{\bibstring{in}\intitlepunct}}}
\usepackage{csquotes}
\usepackage{bbm}
\usepackage{float}
\usepackage{hyperref}
\usepackage{subfig}
\usepackage{caption}
\usepackage{subcaption}
\usepackage{amssymb} 
\usepackage{tikz-cd}
\usepackage{color}

\usepackage[normalem]{ulem}
\usepackage[colorinlistoftodos]{todonotes}
\usepackage{enumitem}

\usepackage{tikz}
\usetikzlibrary{trees}

\newtheorem{theorem}            {Theorem}[section]
\newtheorem{proposition}        [theorem]{Proposition}

\newtheorem{lemma}[theorem]{Lemma}
\newtheorem{definition}[theorem]{Definition}
\newtheorem{remark}{Remark}
\newtheorem{cor}[theorem]{Corollary}

\newcommand{\bbx}{\mathbbm{x}}
\newcommand{\fax}{\langle \bbx \rangle}

\title{Limiting eigenvalue distribution and entropy of multi-Toeplitz matrices}
\author{Michael T. Jury, Arya Gayathri Memana}
\thanks{MJ partially supported by National Science Foundation grant DMS-2154494. A portion of these results appeared in AM's PhD thesis, given at the University of Florida, August 2026.}
\address{Department of Mathematics, University of Florida, Gainesville, FL 32611-8105}
\email{mjury@ufl.edu}
\email{aryagayathmemana@ufl.edu}

\begin{document}

\begin{abstract}
   We state and prove a version of Szeg\H{o}'s first limit theorem for multi-Toeplitz operators acting in the full Fock space in $d\geq 2$ letters. In particular we show that the eigenvalue distributions of truncated multi-Toeplitz operators converge to a limiting distribution. We compute this limiting distribution and its associated entropy in some special cases. Even in the simplest examples the limiting distribution is quite different from the classical case (corresponding to $d=1$); in these examples we obtain a purely atomic measure, which admits a probabilistic interpretation in connection with a percolation process associated to a simple tridiagonal random matrix model. 
\end{abstract}
\maketitle
\section{Introduction}
Szegő's well-known first limit theorem (\cite{MR1544404}) about Toeplitz operators states the following : Consider a bounded positive Toeplitz operator $T_{\phi}$ on the Hardy space with the continuous symbol $\phi: \mathbb{T} \rightarrow \mathbb{R}$. Denote by $\lambda_1, \lambda_2,..\lambda_N$, the $N$ eigenvalues of the $N \times N$ truncation $T_N$ of $T$. Then, for a continuous function $f: \mathbb{R} \rightarrow \mathbb{R}$,
\begin{align}\label{eq11}
    \lim_{N \rightarrow \infty} \frac{1}{N} \sum_{i=1}^{N} f(\lambda_i) = \lim_{N \rightarrow \infty} \frac{1}{N} Tr(f(T_N)) =  \int_{\mathbb{T}} f(\phi) dm.
\end{align}
where $m$ is the normalized Lebesgue measure on $\mathbb{T}$ and $Tr$ is the usual trace of a matrix.\\

Let $P_N$ be the projection onto the linear span of $\{e^{ik \theta}: k=0,1,\cdots,N\}$ on the Hardy space. Then corresponding to each of the operators $P_N TP_N$ we can define a Borel probability measure $\mu_A^N$ on $\mathbb{R}$ as:
\begin{align} 
\mu_T^N(B) = \frac{\nu_T^N(B)}{rank (P_N)} \ \ \ (B \subseteq \mathbb{R}, B \ \text{Borel})
\end{align}
where $\nu_T^N$ is the number of eigenvalues of $P_N T P_N$ in B, counting multiplicities.\\

The spectrum of a Toeplitz operator is the interval $[a,b]$ where $a$ and $b$ are the min and max values of its symbol $\phi$ respectively. In particular, the spectrum of $P_N T P_N$ which is the support of $\mu_T^N$, is contained in $[a,b]$. Equation (\ref{eq11}) can now be read as:

\begin{equation}\label{eqn:szego-limit}
    \lim_{N \rightarrow \infty} \int_{a}^{b} f(x) d\mu_T^N(x) = \int_{a}^{b} f(x) d (\phi_* m)(x)
\end{equation}
where $d(\phi_* m)$ is the push-forward of the Lebesgue measure $m$ under the map $\phi$. In this case, the support of $\mu_*$ is equal to the spectrum of the (self-adjoint) operator $T$. Thus, Szegő’s theorem asserts that the empirical eigenvalue distributions of finite-dimensional truncations of a positive Toeplitz operator converge, in the limit, to a measure determined by its symbol in a natural and explicit way.\\

The Fock space, which can be regarded as a noncommutative analogue of the classical Hardy space, provides a natural setting for studying noncommutative Toeplitz operators. Precise definitions and a review of the relevant results are given in the next section. Noncommutative extensions of classical Toeplitz operators and the Szeg\H{o} limit theorem have been studied by Popescu \cite{popescu-entropy-2001}, and more recently (in connection with noncommutative orthogonal polynomials) by Gauntlett and Kimsey \cite{kimsey-gauntlett}. Both results rely on a free version of the Schur parameter theory developed by Constaninescu and Johnson in \cite{constaninescu-1,constantinescu-2}.

In this paper we have two main goals: first, we give a direct and self-contained proof of the existence of the limiting measure for truncated multi-Toeplitz operators in Theorem~\ref{ex}. This result will be a multi-Toeplitz analog of Szeg\H{o}'s theorem stated above. (This result could be deduced from results in \cite{popescu-entropy-2001}, though it is not quite stated explicitly there. See the discussion in the next section for details.) Second, we turn to the problem of computing this limiting measure in some special cases. One hopes to find a formula something like (\ref{eqn:szego-limit}) expressing the limiting measure in terms of some kind of "symbol" of the multi-Toeplitz operator. Surprisingly, the results differ dramatically from the one-variable case. For example, even for some very simple multi-Toeplitz operators we find that the limiting measure is purely atomic (Theorem~\ref{thm:special-case}, Theorem~\ref{thm:atomic-general}), in contrast to the classical theorem, where the limiting distribution is absolutely continuous with respect to Lebesgue measure on the line. We also use our results to compute some explicit values of the entropy of a multi-Toeplitz kernel (see Definition~\ref{def:entropy}). In the simplest examples we consider, we are able to give a heuristic explanation for the appearance of these atomic measures; it turns out that in these examples the multi-Toeplitz matrices are quite sparse, and it is the sparsity pattern that controls the eigenvalue behavior. However a full picture of the limiting measure for general multi-Toeplitz operators is still lacking.

\section{Preliminaries}
 \begin{definition}
     The Fock space $\mathcal F_d$ on $d$ generators is the Hilbert space direct sum of the tensor powers of $\mathbb C^d$: put $\mathcal H_0=\mathbb C$ and $\mathcal H_n=(\mathbb C^d)^{\otimes n}$ for $n\geq 1$; then 
     \[
     \mathcal F_d= \bigoplus_{n=0}^\infty \mathcal H_n. 
     \]
 \end{definition}
  Fix the standard basis $\{e_1, \dots, e_d\}$ of $\mathbb C^d$.  For a word $w=i_1i_2\cdots i_n$ in the letters $\{1, \dots, d\}$, we let 
 \[
 \xi_w = e_{i_1}\otimes e_{i_2}\otimes \cdots \otimes e_{i_n}\in\mathcal H_n.
 \]
 The $\xi_w$ form an orthonormal basis of $\mathcal F^d$. 
 \begin{definition}
   The left d-shift is defined as the tuple of operators on $L= (L_1,\cdots,L_d)$ where each $L_i$ on $\mathcal{F}_d$ is defined as:
     \begin{align*}
         L_i \xi = e_i \otimes \xi
     \end{align*}
for $\xi\in\mathcal H_n$. Similarly the right d-shift $R=(R_1, \dots, R_d)$ is defined by
\[
R_i\xi =\xi\otimes e_i
\]
 \end{definition}
The $L_i$ form a system of isometries with orthogonal ranges, that is, $L_j^*l_i =\delta_{ij}I$. 

\begin{definition}
An operator $T$ is said to be {$R-$}{multi-Toeplitz} if $R_i^* T R_j = \delta_{ij}T$ and {$L-$}{multi-Toeplitz} if $L_i^* T L_j = \delta_{ij}T$ .
\end{definition}

 Note that $L_w$ and $L_w^*$ are $R-$multi-Toeplitz for every word $w$. And so in general any polynomial of the form,
 
 \begin{align*}
     p(L) + q(L)^*
 \end{align*}
is $R-$multi-Toeplitz. Here $p$ and $q$ are polynomials in $d$ free indeterminates $x_1,\cdots, x_d$.\\

From here on we say multi-Toeplitz to mean $R-$multi-Toeplitz, but all of our results have obvious analogs for $L-$multi-Toeplitz operators. In this paper we focus on positive multi-Toeplitz operators of the form,

 \begin{align}\label{eqmt}
     T = q_0+ \sum_{|w|\leq M} q_w L_w + \sum_{|w|\leq M} \overline{q_w} L_w^*
 \end{align}
The matrix representation of the above operator with respect to the Fock space basis $\{\xi_w\}$ has the form:
\begin{equation}\label{eqn:big-toeplitz}
T =\begin{bmatrix}
        q_0 & q_{-1}  & \cdots & q_{-n} & 0 & \cdots\\
        q_{1} & q_0\otimes I_d &  \cdots & q_{-(n-1)} \otimes I_d  & q_{-n} \otimes I_d & \cdots\\
         q_{2} & q_{1} \otimes I_d & \cdots &  q_{-(n-2)} \otimes I_d \otimes I_d  & q_{-(n-1)} \otimes I_d \otimes I_d  & \cdots\\
        \vdots & \vdots &  \ddots & \vdots  & \vdots  & \ddots\\
         q_{n} & q_{n-1} \otimes I_d &  \cdots &  q_0 \otimes I_d \otimes \cdots\otimes I_d  & q_{-1} \otimes I_d^{\otimes n}  & \ddots\\
         0 & q_{n} \otimes I_d & \cdots &  q_1 \otimes I_d^{\otimes n}  & q_{0} \otimes I_d^{\otimes (n+1)} & \ddots\\
         0 & 0 &  \cdots &  q_2 \otimes I_d^{\otimes n}  & q_{1} \otimes I_d^{\otimes (n+1)} & \ddots\\
         \vdots & \vdots & \vdots & \ddots & \ddots  &\ddots \\

      \end{bmatrix}.\end{equation}

      \vspace{0.4cm}

where \[q_k\::=\: col(q_w)_{w\in\fax ;|w|=k} \,\mbox{    and    }\, q_{-k}\::=\: row(q_w^*)_{w\in\fax ;|w|=k},\]

and 
\[q_1 \otimes I_d:=\begin{bmatrix}
             q_1 & \cdots & 0\\
             \vdots & \ddots & \vdots\\
             0 & \cdots & q_1\\
             \vdots & \ddots & \vdots\\
             \vdots & \ddots & \vdots\\
             q_d & \cdots & 0\\
             \vdots & \ddots & \vdots\\
             0 & \cdots & q_d\\
 \end{bmatrix}
 \]
(See for example \cite{MR4577945}). We use the term \textit{symbol} to mean the noncommutative polynomial associated with the multi-Toeplitz operator. So in $(\ref{eqmt})$, $T$ has the polynomial symbol $q(x)= q_0 + \sum_{0 \leq |w| \leq N} q_w \xi_w + \sum_{0 \leq |w| \leq N} \bar{q_w} x^{w^*} $. (In the classical case this corresponds to the usual notion of the symbol of a Toeplitz operator, when that symbol is a trigonometric polynomial. Note that we are considering only the polynomial case, and do not define the "symbol" of a multi-Toeplitz operator in general.)
  
With an eye towards a Szeg\H{o}-style limit theorem, the natural choice of projections here is the following: for $N \in \mathbb{N}$, we let $P_N$ denote the projection of $\mathcal{F}_d$ onto words of length at most $N$. i.e, projection onto the subspace spanned by the monomials $\{\xi_w: |w| \leq N\}$. These projections have finite rank which can be computed easily as $rank(P_N)=:d_N= 1+d+d^2+ \cdots d^N$.  Denote by $T_N=P_NTP_N$ the truncations of  the multi-Toeplitz operator at size $d_N$. Our first theorem will say that the eigenvalue distributions, 
 \begin{align*}
     \mu_{T_N} : = \frac{1}{d_N} (\delta_{\lambda_1} + \cdots+ \delta_{\lambda_{d_N}}).
 \end{align*}
 where $\{\lambda_i\}_1^{d_N}$ are the $d_N$ eigenvalues of $T_N$, have a weak  limit $\mu_T.$ 
 
 Popescu in \cite{popescu-entropy-2001}, introduces the notion of entropy for positive-definite multi-Toeplitz kernels on the Fock space $\mathcal F_d$. 

A \textit{positive-definite kernel} on $\mathcal{F}_d$ is a map $K : \mathcal{F}_d \times \mathcal{F}_d \rightarrow B(H) $ with the property that
\[
\sum_{i,j=1}^k \langle K(\sigma_i, \sigma_j) h_j, h_i \rangle \geq 0
\]
for all choices of vectors $h_1, h_2, \cdots, h_k \in H$, words $\sigma_1, \sigma_2, \cdots, \sigma_k \in \mathcal{F}_d$ and $k \in \mathbb{N}$. A kernel $K$ on $\mathcal{F}_d$ is called multi-Toeplitz if $K(e,e) = I_{H}$ ($e$ is the neutral element of $\mathcal{F}_d$) and
\[
K(\sigma, w)= \begin{cases} 
      K(\alpha, e); & \text{ if $\sigma = w \alpha$ for some $\alpha  \in \mathcal{F}_d$} \\
     K(e, \alpha)& \text{if $w= \sigma \alpha $ for some $\alpha \in \mathcal F_d$} \\
      0 & \text{otherwise}
   \end{cases}
\]
It can be verified that the positive multi-Toeplitz operators on $\mathcal{F}_d$ correspond to positive kernels on $\mathcal{F}_d$ via,
\[
K_T( \sigma, w) : = \langle T \sigma, w \rangle \quad \sigma, w \in \mathcal F_d
\]
The entropy $E(K)$ of $K$ is defined by:
\begin{definition}\label{def:entropy}\cite[Equation 2.12]{popescu-entropy-2001}
    \[\exp{E(K)}: = \lim_{N \rightarrow \infty} \left(  \frac {det M_N} { det M_{N-1}} \right)^{1/d^N}\]
    where $M_N$ is the truncation of the multi-Toeplitz kernel matrix at block size $N$.
\end{definition}
Then we have the following result \cite[Corollary 2.4]{popescu-entropy-2001}:
\begin{proposition} \label{popescu}
    Assume $K$ is a positive definite multi-Toeplitz kernel and $det M_N \neq 0$ for any $N=1,2,3,\cdots$. Let $\lambda_1, \lambda_2, \cdots, \lambda_q$ be the eigenvalues of $M_N$. Then,
    \begin{equation}\label{eqn:entropy}
        \lim_{N \rightarrow \infty} \frac{\log \lambda_1 + \log \lambda_2 + \cdots + \log \lambda_q}{d_N} = E(K)
    \end{equation}
    where  $E(K)$ is the entropy of $K$ and $d_N = 1+ d + d^2 + \cdots+ d^N$.
\end{proposition}
It is possible to deduce from this statement, along the lines of Szeg\H{o}'s original theorem, that if one makes the additional assumption that the operator determined by $K$ is bounded, then the $M_N$ have a limiting eigenvalue distribution, although this is not carried out explicitly in \cite{popescu-entropy-2001}. When $K$ comes from a bounded Toeplitz operator $T$, the $M_N$ in Popescu's notation are our truncations $T_N=P_NTP_N$. We now give a self-contained proof of the existence of the limiting eigenvalue distribution for truncations of bounded, self-adjoint multi-Toeplitz operators. 
 
\begin{theorem}\label{ex}
 Let $T$ be a self-adjoint, bounded  multi-Toeplitz operator in $d\geq 2$ variables. For each $T_N= P_N TP_N$, let $\lambda_1^{(N)}, \lambda_2^{(N)}, \cdots, \lambda_{d_N}^{(N)}$, be its eigenvalues. There exists a unique probability measure $\mu_T$ on $\mathbb{R}$, such that for every continuous function $f: \mathbb{R} \rightarrow \mathbb{C}$, we have,
        
        \begin{align}\label{eq15}
            \lim_{N \rightarrow \infty} \frac{1}{d_N} \sum_{i=1}^{d_N} f(\lambda_i^{(N)}) = \int_{\mathbb{R}} f\,  d\mu_T
        \end{align}
        Moreover, if $[a,b]$ is the smallest interval containing the spectrum of $T$, then $\mu_T$ is supported in $[a,b]$.
    \end{theorem}

\begin{proof}
   By adding a suitable constant multiple of the identity to $T$, we may assume $T$ is positive.  We then observe that if $[a,b]$ is any interval which contains the spectrum of $T$, it will also contain the spectrum of $T_N$. It follows that for any continuous function $f$ on $[a,b]$, we have the estimate
   \begin{equation}\label{eqn:uniform-estimate}
   \frac{1}{d_N} \left|\sum_{j=1}^{d_N} f(\lambda_j^{(N)}) \right| \leq \sup_{x\in [a,b]} |f(x)|
   \end{equation}
It will therefore suffice to prove that the limit in the left-hand side of $\ref{eq15}$ exists for $f(x)=x^p, p=0,1,\cdots,$. Indeed, if this is the case, then the limit will exist for all polynomials by linearity, and then also for all continuous $f$ by (\ref{eqn:uniform-estimate}) and uniform approximation. In other words, we will have proved that the map sending $f$ to this limit defines a continuous linear functional on $C[a,b]$, and then we conclude $\mu_T$ exists by appeal to the Riesz-Markov-Kakutani theorem. (In particular this shows that if $[a,b]$ is the smallest interval containing the spectrum of $T$, then $\mu_T$ is supported in $[a,b]$. )

    So fix $p\geq1$ and put
    \begin{align*}
        \alpha_N = (\lambda_1^{(N)})^p + (\lambda_2^{(N)})^p + \cdots+ (\lambda_{d_N}^{(N)})^p.
    \end{align*}
It will suffice to prove that $\lim_{N\to \infty}  \left( \frac{\alpha_N}{d^N} \right)^{1/p}$ exists. To do this, observe first that $\alpha_N^{1/p}$ is nothing but the Schatten $p$-norm $\|T_N\|_p$. Let us compare $\|T_{N+1}\|_p$ with $\|T_N\|_p$. From the multi-Toeplitz structure (\ref{eqn:big-toeplitz}), we see that $T_{N+1}$ is obtained from $T_N$ by first replacing $T_N$ with $T_N \otimes I_d $, and then appending a new first row and column to this matrix, where the new upper left corner entry is $q_0$ and the new row and column are the portions of the first row and column of the full (infinite) matrix for $T$ up to words of length $N+1$ 
From the definition of Schatten norms we have $\|T_N \otimes I_d\|_p = d^{1/p} \|T_N\|_p$, and appending the new row and column can perturb the $p$-norm by only a bounded amount (independent of $N$), namely the bound is the $p$-norm of the matrix consisting of the appended row and column, and the rest $0$'s. In summary we have,
\begin{align*}
    \|T_{N+1}\|_p = d^{1/p} \|T_N\|_p + O(1).
\end{align*}

Dividing through by $d^{(N+1)/p}$, we obtain
\begin{align*}
    \left( \frac{\alpha_{N+1}}{d^{N+1}} \right)^{1/p} = \left( \frac{\alpha_N}{d^N}  \right)^{1/p} + O\left(\frac{1}{d^{(N+1)/p}} \right).
\end{align*}
Thus, (since $d\geq 2$)! the differences between successive terms are summable in $N$, and we conclude that $(\frac{\alpha_N}{d^N})^{1/p}$ is a Cauchy sequence, hence convergent.
\end{proof}

\begin{remark}
    Note that this proof \textit{does not} work when $d=1$. Indeed in that case rather than the geometric decay, we conclude only that the gaps between successive terms decay at least as fast as $1/N$, which of course isn't enough to guarantee convergence.the convergence of the roots of the determinants is not sufficient, all by itself, to get a distance formula. 
\end{remark}
As was already noted, when $K=K_T$ then determinants $\det M_N$ are exactly $\det T_N$. For each $N$ the quantity in the left hand side of (\ref{eqn:entropy}) is
\[
\int_{\mathbb R} \log x \, d\sigma_N
\]
where $\sigma_N$ is the empirical eigenvalue distribution of $T_N$. 
So in particular, for $T$ a multi-Toeplitz operator on $\mathcal F_d$, we have
\begin{cor}\label{cor:entropy-fromula}
Let $T$ be a positive definite multi-Toeplitz operator acting in $\mathcal F_d$ and let $K_T$ be the associated positive kernel. Then the Popescu entropy $E(K_T)$ is given by the formula
\[
E(K_T) = \int_{\mathbb R} \log x\, d\mu_T(x)
\]
where $\mu_T$ is the limiting spectral distribution guaranteed by Theorem~\ref{ex}. 
\end{cor}
The proof of the above theorem relies heavily on the structure of the associated multi-Toeplitz matrices and guarantees the existence of a Szegő-type limiting measure. As described in the introduction, in the classical setting the limiting measure is straightforward to describe, as the push-forward of Lebesgue measure on the circle, under the symbol function. In contrast, identifying this limiting measure in our noncommutative setting turns out to be rather more delicate. Unlike in the classical and Drury–Arveson settings, neither the classical arguments nor the more general $C^*$-algebraic approach is available here, since the projections $P_N$'s
  fail to form a Følner sequence. (See \cite{memana-2026} for a discussion of these approaches and their application in the Drury-Arveson space and related settings.) In the next section we compute the limiting measure explicitly in some concrete cases, and observe a stark contrast with the one dimensional situation. 

\section{Examples of the limiting eigenvalue distribution}

For $h\geq 1$, let $X_h$ be the $h+1\times h+1$ matrix whose $(i,j)$ entry is $1$ if $|i-j|=1$ and $0$ otherwise (a Jordan block plus its adjoint). Let $\mu_h$ be the spectral measure of $X_h$. It is well known and straightforward to prove that 
\[
\mu_h =\sum_{j =1}^{h+1}\delta_{\alpha_{j,h}}
\]
where $\alpha_{j,h} =2\cos\frac{j\pi}{h+2}$. Note that $\mu_h$ is symmetric about the origin.

\begin{theorem}\label{thm:special-case}
    Let $T=L_1+L_1^*$ act in the full Fock space $\mathcal F_d$, $d\geq 2$, and let $\mu_T$ be the limiting eigenvalue distribution of the truncations $T_N$ as in Theorem~\ref{ex}. Then $\mu_T$ is a countable sum of atoms, given explicitly by
\begin{equation}\label{eqn:muT}
     \mu_T= \left(1-\frac1d\right)^2 \sum_{h=0}^\infty  \left(\frac1d\right)^h \mu_h. 
\end{equation}
\end{theorem}
Before proving the theorem, we make a few comments on the structure of this limiting measure. First, as a sanity check, it can be verified easily that $\mu_T$ is a probability measure supported in $[-2,2]$. Indeed, note that each $\mu_h$ has total mass $h+1$. Then, for any $0<t<1$,
\[
(1-t)^2 \sum_{h=0}^\infty (h+1) t^h =1.
\]
The atoms $\{2 cos(\frac{\pi j}{h+2}): j= 1,\cdots, h+1\}_{h \geq 0}$ evidently all lie in  $[-2,2]$. This calculation also shows that the measure $\mu_T$ belongs to a continuous family of measures parameterized by $t$ for $0 <t < 1$. The measure $\mu_T$ corresponds to $t= 1/d$. As the parameter $t$ approaches $1$, the measures converge weak-* to the arcsine distribution. This agrees with the classical case of the Toeplitz operator $S+S^*=T_{z+\overline{z}}$, where the limiting eigenvalue distribution is the arcsine law. This behavior is illustrated in Figure~\ref{fig:cdfs}.
\begin{figure}[H]
\centering

\begin{minipage}{0.4\textwidth}
\centering
\includegraphics[width=\textwidth]{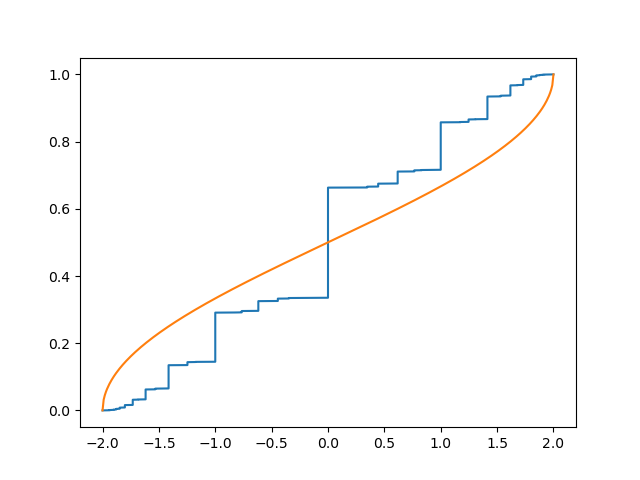}\\
{\small (i) $t=0.5$}
\end{minipage}
\hfill
\begin{minipage}{0.4\textwidth}
\centering
\includegraphics[width=\textwidth]{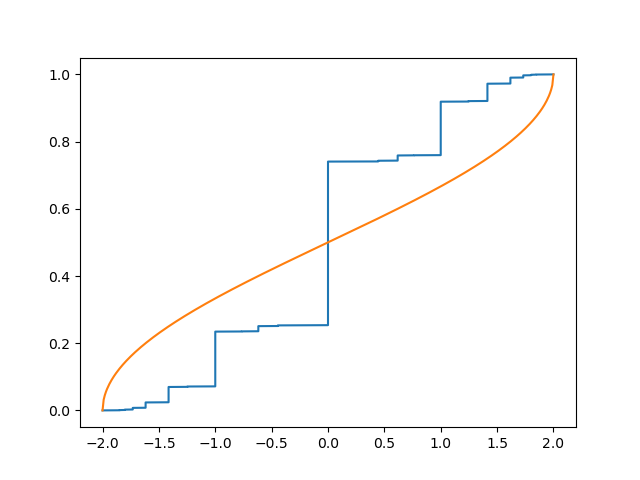}\\
{\small (ii) $t=0.333\dots$}
\end{minipage}

\medskip

\begin{minipage}{0.4\textwidth}
\centering
\includegraphics[width=\textwidth]{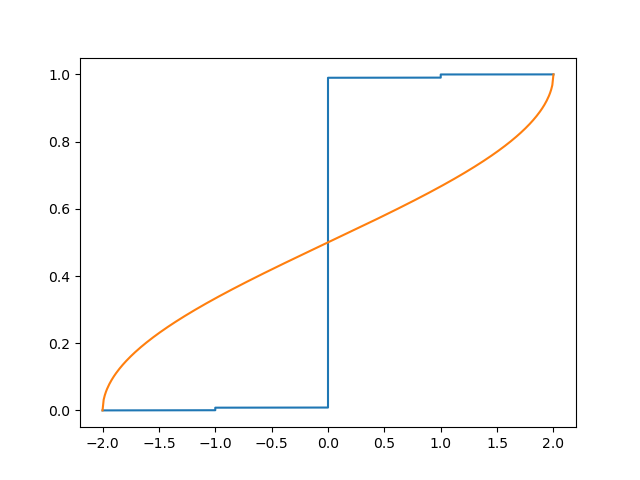}\\
{\small (iii) $t=0.01$}
\end{minipage}
\hfill
\begin{minipage}{0.4\textwidth}
\centering
\includegraphics[width=\textwidth]{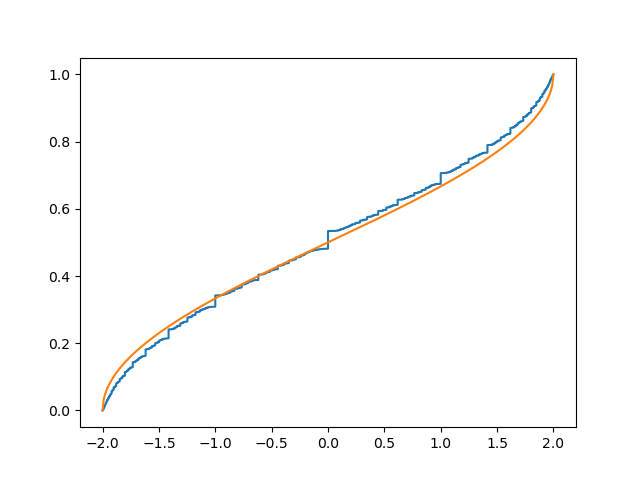}\\
{\small (iv) $t=0.9$}
\end{minipage}

\caption{Comparison of the arcsine law with the cumulative distribution functions for $\mu_T$, where $T = L_1 + L_1^*$ with different weights $t= \frac{1}{d}$, for $d=2,3, 100$ and the ``non-physical" value $d=10/9$}\label{fig:cdfs}
\end{figure}

\begin{proof}[Proof of Theorem~\ref{thm:special-case}]
Consider the truncated multi-Toeplitz operator $P_N(L_1+L_1^*)P_N$. Observe that for any words $v,w$ 
\[
\langle (L_1+L_1^* )\xi_v, \xi_w\rangle =0
\]
unless $v=1w$ or $w=1v$, in which case the inner product is $1$. Thus, the matrix of $T_N$ is the incidence matrix of a graph whose nodes are labeled by words of length at most $N$, with an edge between node $v$ and node $w$ if and only if either $v=1w$ or $w=1v$. Consequently, the resulting graph consists of $d^{N-1}$ isolated nodes (corresponding to the words $v$ of length $N$ which start with a letter different from $1$), and the remainder of the graph splits into connected components, each of which is a linear chain. (See Figure~\ref{fig:tree}.) The chain may be labeled by the longest word it contains, which always has length $N$ and starts with a $1$, and hence is a word of the form $1^hv$ where $1\leq h\leq N$, and $v$ is a word of length  $|v|=N-h$ which starts with a letter distinct from $1$. The length of such a chain is evidently $h$, and the incidence matrix thus decomposes into a $0$ summand, and a direct sum of copies of $X_h$, each copy of $X_h$ appearing as many times as the number of chains of length $h$ in the graph. From the above description, for $1\leq h<N$ this number is $(d-1) d^{N-h-1}$ (the number of words of length $N-h$ which do not start with $1$), and there is one chain of length $N$. We conclude that the spectral measure of $T_N$ is the sum
\begin{equation}\label{eqn:esd-N}
\mu_N+\sum_{h=1}^{N-1} (d-1)d^{N-h-1} \mu_h.
\end{equation}
Now the size of the matrix $T_N$ is 
\[
1+d+d^2+\cdots +d^N = \frac{d^{N+1}-1}{d-1}
\]
so the empirical spectral measure of $T_N$ is 
\begin{align*}
\sigma_N &= \frac{d-1}{d^{N+1}-1} \mu_N + \sum_{h=1}^N \frac{(d-1)^2d^{N-h-1}}{d^{N+1}-1} \mu_h\\
&= \frac{d-1}{d^{N+1}-1} \mu_N + \left(\frac{d-1}{d}\right)^2\frac{1}{1-d^{-(N+1)}} \sum_{h=1}^N \frac{1}{d^h} \mu_h
\end{align*}
Since the infinite series $\sum_{h=1}^\infty \frac{1}{d^h} \mu_h$ is summable in the total variation norm (note that the total mass of $\mu_h$ is $h+1$), it follows easily that as $N\to \infty$ the $\sigma_N$ converge in total variation norm to the measure
\[
\mu_T = \left(1-\frac1d\right)^2 \sum_{h=1}^\infty d^{-h}\mu_h. 
\]
\end{proof}

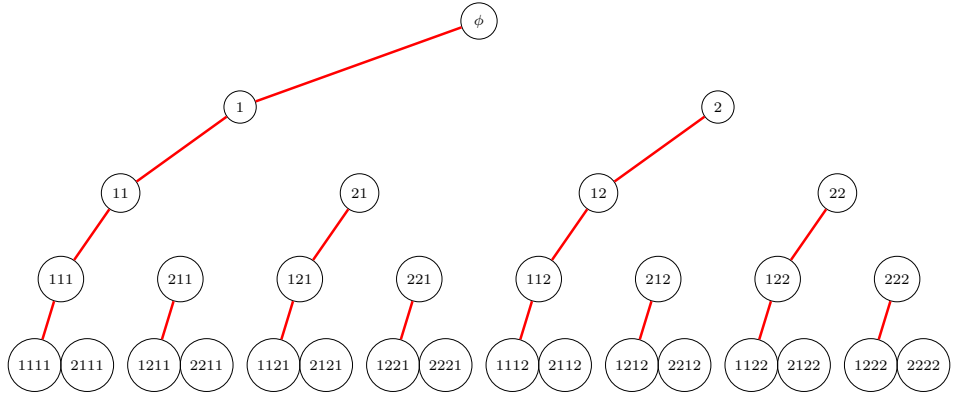
\begin{figure}[h]
\centering
\resizebox{\textwidth}{!}{
\begin{tikzpicture}[
  scale=0.8,
  level distance=1.8cm,
  level 1/.style={sibling distance=10cm},
  level 2/.style={sibling distance=5cm},
  level 3/.style={sibling distance=2.5cm},
  level 4/.style={sibling distance=1.1cm},
  every node/.style={circle,draw,minimum size=4mm, font=\scriptsize},
  edge from parent/.style={draw=none} 
]

\node (phi) {$\phi$}
  child {node(n1) {1}
    child {node(n11) {11}
      child {node(n111) {111}
        child {node(n1111) {1111}}
        child {node(n2111) {2111}}
      }
      child {node(n211) {211}
        child {node(n1211) {1211}}
        child {node(n2211) {2211}}
      }
    }
    child {node(n21) {21}
      child {node(n121) {121}
        child {node(n1121) {1121}}
        child {node(n2121) {2121}}
      }
      child {node(n221) {221}
        child {node(n1221) {1221}}
        child {node(n2221) {2221}}
      }
    }
  }
  child {node(n2) {2}
    child {node(n12) {12}
      child {node(n112) {112}
        child {node(n1112) {1112}}
        child {node(n2112) {2112}}
      }
      child {node(n212) {212}
        child {node(n1212) {1212}}
        child {node(n2212) {2212}}
      }
    }
    child {node(n22) {22}
      child {node(n122) {122}
        child {node(n1122) {1122}}
        child {node(n2122) {2122}}
      }
      child {node(n222) {222}
        child {node(n1222) {1222}}
        child {node(n2222) {2222}}
      }
    }
  };

\draw[red, very thick] (phi) -- (n1);
\draw[red, very thick] (n1) -- (n11);
\draw[red, very thick] (n11) -- (n111);
\draw[red, very thick] (n21) -- (n121);
\draw[red, very thick] (n2) -- (n12);
\draw[red, very thick] (n112) -- (n12);
\draw[red, very thick] (n22) -- (n122);
\draw[red, very thick] (n111) -- (n1111);
\draw[red, very thick] (n211) -- (n1211);
\draw[red, very thick] (n121) -- (n1121);
\draw[red, very thick] (n221) -- (n1221);
\draw[red, very thick] (n112) -- (n1112);
\draw[red, very thick] (n212) -- (n1212);
\draw[red, very thick] (n122) -- (n1122);
\draw[red, very thick] (n222) -- (n1222);

\end{tikzpicture}

}
\caption{Deterministic percolation process given by the truncated multi-Toeplitz matrices $T_N$ for $T=L_1+L_1^*$}\label{fig:tree}
\end{figure}

The argument can be easily adapted to prove the following more general result: 
\begin{theorem}\label{thm:atomic-general}
    Let $f:\mathbb T\to \mathbb C$ be a continuous analytic function with Fourier coefficients $c_j=\widehat{f}(j)$, $j\geq 0$.
    Consider the multi-Toeplitz operator $T=f(L_1)+f(L_1)^*$ acting in $\mathcal F_d$, $d\geq 2$. Let $T_N=P_NTP_N$ and for $h=0, 1, 2, \dots$ let $X_h^f$ be the $h+1\times h+1$ Toeplitz matrix
   \[ \begin{pmatrix}
        2c_0 & \overline{c_{1}} & \cdots & \overline{c_{h}}\\ c_1 & 2c_0 &\cdots & \overline{c_{h-1}} \\ \vdots & \ddots  & \ddots & \vdots\\ c_h & c_{h-1} & \cdots & 2c_0
    \end{pmatrix}.
   \] (These are truncations of the Toeplitz operator with symbol $f+\overline{f}$.) Let 
    \[
    \mu_h^f =\sum_{j=1}^h \delta_{\lambda_j^{(h)}}
    \]
    be the atomic measure with point masses at the eigenvalues $\lambda_j^{(h)}$ of $X_h^f$. Then the limiting eigenvalue distribution $\mu_T$ of the truncations $T_N$ is purely atomic, and is given by
    \[
    \mu_T= \left( 1-\frac1d\right)^2 \sum_{h=0}^\infty d^{-h} \mu_h^f.
    \]
    
\end{theorem}
\begin{proof}
    Let $v$ be a word of length at most $N$ which does not start with $1$ (the empty word is allowed), and let $\mathcal H_v$ be the span of the basis vectors $\xi_w$ over all words of the form $1^kv$ where $k+|v|\leq N$. (Note that if $|v|=N$ then $\mathcal H_v$ is one-dimensional, spanned by $\xi_v$). It is easily checked that each of the subspaces $\mathcal H_v$ is reducing for $T_N$, and by examining the inner products
    \[
    \langle c_j L^j \xi_{1^kv}, \xi_{1^\ell v}\rangle
    \]
    one sees that the restriction of $T_N$ to this subspace is unitarily equivalent to $X_h^f$ where $h=N-|v|$. Thus, $T_N$ splits as a direct sum of copies of $X_h^f$, the number of copies equal to the number of words $v$ such that $\mathcal H_v$ has dimension $h+1$. But this is exactly the number of chains of length $h$ as counted in the previous proof. It follows that the spectral measure of $T_N$ is
    \[
    \sigma_N^f := \mu_N^f + \sum_{h=0}^{N-1} (d-1) d^{N-h-1}\mu_h^f.
    \]
    The remainder of the proof proceeds as before. 
\end{proof}

We can extend Theorem~\ref{thm:special-case} to replace $L_1$ with a homogeneous sum $\sum_{|w|=M} a_w L^w$; this is accomplished using the following two lemmas. 

\begin{lemma}\label{lem:one-letter}
    Let $a_1, \dots, a_d$ be scalars with $\sum_{j=1}^d |a_j|^2=1$. There is a unitary transformation $U:\mathcal F_d\to \mathcal F_d$ such that
    \[\sum_{j=1}^d a_j L_j = UL_1U^*\]
    and
    \[UP_N=P_NU\]
    for each $N$. 
\end{lemma}
\begin{proof}
   Let $\{e_1, \dots, e_d\}$ be the standard orthonormal basis for $\mathbb C^d$ and let $f_1=\sum_{j=1}^da_je_j$. Let $\{f_1, \dots, f_d\}$ be an orthonormal basis of $\mathbb C^d$ extending $\{f_1\}$, and let $V$ be any unitary map of $\mathbb C^d$ taking $e_1$ to $f_1$. The map $V$ then determines a unitary $V^{\otimes n}$ acting in $(\mathbb C^d)^{\otimes n}$, taking the direct sum of these gives a unitary $U$ acting in $\mathcal F_d$. By construction $U$ commutes with $P_N$.    Now for any vector $\xi_n\in(\mathbb C^d)^{\otimes n}$, we have
    \begin{align*}
  (UL_1U^*) \xi_n &=U(e_1\otimes (V^{\otimes n})^* \xi_n)\\
  &=(V\otimes (V^{\otimes n}))[e_1\otimes (V^{\otimes n})^*\xi_n] \\
  &= f_1\otimes \xi_n\\
  &=\left( \sum_{j=1}^da_j L_j\right)\xi_n.
    \end{align*}
    
\end{proof}
\begin{lemma}\label{lem:long-words} Fix an integer $M\geq 1$. Let $L_1, \dots, L_d$ and $\Lambda_1, \dots, \Lambda_\delta$ denote the shifts on $d$ and $\delta=d^M$ letters respectively. Let $P_N$ denote the projection in $\mathcal F_d$ onto the span of words of at most $N$ letters, and similarly $\Pi_{N^{\prime}}$ the projection in $\mathcal F_{\delta}$. Let $\{a_w\}$ be scalars indexed by the words $w$ of length $M$ in $d$ letters, and let $\alpha_j$ be the same scalars as the $a_w$, labeled using a fixed bijection of the set of words of length $M$ in $d$ letters, with the set $\{1,\dots, \delta\}$.

Let $T=\sum_{|w|=M}\left(a_wL^w+\overline{a_w}L^{w*}\right)$ and $\widetilde{T} = \sum_{j=1}^\delta\left(\alpha_j\Lambda^j+\overline{\alpha_j}\Lambda^{j*}\right)$. 

Then for each $N^\prime \geq 1$, putting $N =MN^\prime-1$, the operator $T_{N} = P_{N}TP_{N}$ is unitarily equivalent to a direct sum of $1+d+\cdots +d^{M-1}$ copies of $\widetilde{T}_{N^\prime} = \Pi_{N^\prime} \widetilde{T} \Pi_{N^\prime}$.  

\end{lemma}
\begin{proof}
    Fix $N$ and $N^\prime$ as in the statement. For each word $v$ in $d$ letters, of length at most $M-1$, let $\mathcal H_v\subset \mathcal F^d$ be the subspace
    \[
    \mathcal H_v = \text{span}\{ \xi_{w_1w_2\cdots w_nv} : |w_i|=M, \, 0\leq n\leq N^\prime-1\}.
    \]
    The spaces $\{\mathcal H_v\}$ are mutually orthogonal, and reducing for $T_N$. The direct sum of the $\mathcal H_v$ over the words $v$ of length at most $M-1$ is equal to the range of the projection $P_{N}$. It is  then straightforward to check that the compression of $T_{N}$ to $\mathcal H_v$ is unitarily equivalent to $\widetilde{T}_{N^\prime}$. Since the number of subspaces $\mathcal H_v$ is equal to $1+d+\cdots +d^{M-1}$, the lemma follows. 
\end{proof}

\begin{theorem}\label{thm:general-case}
    Keeping the notation of the previous lemmas, we impose the normalization $\sum_{|w|=M}|a_w|^2=1$. Then the limiting eigenvalue distribution of the truncations $T_N$ of the multi-Toeplitz operator
    \[
    T=\sum_{|w|=M} a_wL_w+\overline{a_w}L^{w*}
    \]
acting in the Fock space $\mathcal F_d$ over $d$ letters, is equal to the limiting eigenvalue distribution of the truncations of $L_1+L_1^*$ acting in the Fock space over $d^M$ letters, namely
    \[
    \mu_T= \left(1-\frac{1}{d^{M}}\right)^2 \sum_{h=0}^\infty \left(\frac{1}{d^M}\right)^h\mu_h
    \]
with $\mu_h$ as in Theorem~\ref{thm:special-case}.
\end{theorem}
\begin{proof}
    By Lemma~\ref{lem:long-words}, $T_N$ is unitarily equivalent to a direct sum of $d_{M-1}=1+d+\cdots d^{M-1}$ copies of $\widetilde{T}_{N^\prime}$, with $N=MN^\prime-1$. The rank of the projection $P_N$ is \[d_{MN^\prime-1}=1+d+d^2+\cdots + d^{MN^\prime-1},\] and the rank of the projection $\Pi_{N^\prime}$ is \[(d^M)_{N-1}= 1+d^M +d^{2M} +\cdots + d^{M(N-1)}.\] For each $k\geq 1$, we obtain
    \begin{align*}
    \frac{1}{rank P_N}&Tr (P_NTP_N)^k= \frac{1}{1+d+\cdots + d^{MN^\prime-1}}Tr (P_NTP_N)^k \\
    &=   \frac{1+d+\cdots +d^{M-1}}{1+d+\cdots +d^{MN^\prime-1}} Tr (\Pi_{N^\prime}\widetilde{T}\Pi_{N^\prime})^k\\
    &=  \frac{1+d+\cdots +d^{M-1}}{1+d+\cdots +d^{MN^\prime-1}} \frac{1+d^M+\cdots +d^{M(N^\prime-1)})}{rank \Pi_{N^\prime}}Tr (\Pi_{N^\prime}\widetilde{T}\Pi_{N^\prime})^k\\
    \end{align*}
    But now
    \[
    \lim_{N^\prime\to\infty} \frac{1+d+\cdots +d^{M-1}}{1+d+\cdots +d^{MN^\prime-1}}\cdot(1+d^M+\cdots +d^{M(N^\prime-1)})=1
    \]
    so the two systems of truncations have the same limiting moments, and hence the same limiting eigenvalue distribution. That is, the limiting eigenvalue distribution of the original $T_N$ is equal to the limiting eigenvalue distribution of the truncations of the operator
    \[
    \widetilde{T}= \sum_{j=1}^\delta \alpha_j \Lambda_j +\overline{\alpha_j}\Lambda_j^*,
    \]
    a multi-Toeplitz operator acting in the Fock space over $\delta=d^M$ letters. But by Lemma~\ref{lem:one-letter}, the truncations of this operator are unitarily equivalent to the truncations of $L_1+L_1^*$ acting in the $d^M$-letter Fock space. The conclusion now follows from Theorem~\ref{thm:special-case}.\\
\end{proof}

\section{Szeg\H{o} Entropy}
In this section we use our results to compute the entropy 
\[
E(K_T)=\int_{\mathbb R} \log x\, d\mu_T(x)
\]
in a special case. Under the assumption that $T$ is positive definite, $\mu_T$ (as well as the empirical spectral distribution $\sigma_N$, for each $N$) is supported in an interval $[a,b]$ for some $a>0$, so the logarithm is continuous on this interval. It is not immediate that the above formula remains valid when the support of the limiting measure includes $0$, but this will in fact be true in the special case $T=2+L_1+L_1^*$, as we now show. 

    Consider the positive multi-Toeplitz operator $T = 2+ L_1+ L_1^*$
    acting on the Fock space $\mathcal{F}_2$.
The Szeg\H{o} limiting measure of $T$ is given by:
\[
     \mu_T= \left(1-\frac12\right)^2 \sum_{h=0}^\infty \left(\frac12\right)^h \widetilde{\mu_h}
\]
as in Theorem~\ref{thm:special-case}, where now $\widetilde{\mu_h}$ is the translate of $\mu_h$ by $2$ units to the right:
\[
\widetilde{\mu_h} (E) := \mu_h(-2+E).
\]
Let us first fix $h$ and compute $\int \log x\, d\mu_h(x)$. By definition this quantity is 
\[
\sum_{j=1}^{h+1} \log\left( 2+\cos \frac{\pi j}{h+2}\right).
\]
Making the change of variable $j\to h+2-j$ and using the symmetry $\cos \theta = -\cos (\pi-\theta)$, this is seen to be equal to
\begin{align*}
\frac12 \sum_{j =1}^{h+1}  \log \left( 2+ 2 \cos{\frac{ \pi j}{h+2}}\right)+ \log\left(2- 2 \cos{\frac{ \pi j}{h+2}}   \right) &= \frac12 \sum_{j=1}^{h+1} \log\left( 4 \sin^2 \frac{ \pi j}{h+2}\right)\\
&= \log \left(2^{h+1}\prod_{j=1}^{h+1} \sin\frac{\pi j}{h+2}\right)\\
&= \log(h+2)
\end{align*}
where we have used the identity \cite[Formula 1.392(1)]{gr-tables}
\[
\prod_{j=1}^{h+1} \sin \frac{\pi j}{ h+2} = \frac{h+2}{2^{h+1} }.
\]
(As an alternative proof, one could verify by induction that $\det (2I+X_h) = h+2$.)
Using (\ref{eqn:esd-N}), we conclude that 
\[
\int_{\mathbb R} \log x \, d\sigma_N = \frac{d-1}{d^{N+1}-1} \log (N+2) +\left( \frac{d-1}{d}\right)^2 \frac{1}{1-d^{-(N+1)}} \sum_{h=1}^N \frac{1}{d^h}\log (h+2)
\]
It is straightforward to take the limit as $N\to \infty$ to obtain
\[
E(K_T)=\int_{\mathbb R} \log x\, d\mu_T = \left(\frac{d-1}{d}\right)^2\sum_{h=1}^\infty d^{-h}\log(h+2).
\]
We numerically computed the first 20 partial sums (see Figure~\ref{fig:placeholder}) in the case $d=2$ to obtain the estimate
\[
E(K_T) \approx {0.507}. 
\]

\begin{figure}[H]
\centering
    \includegraphics[width=0.7\linewidth]{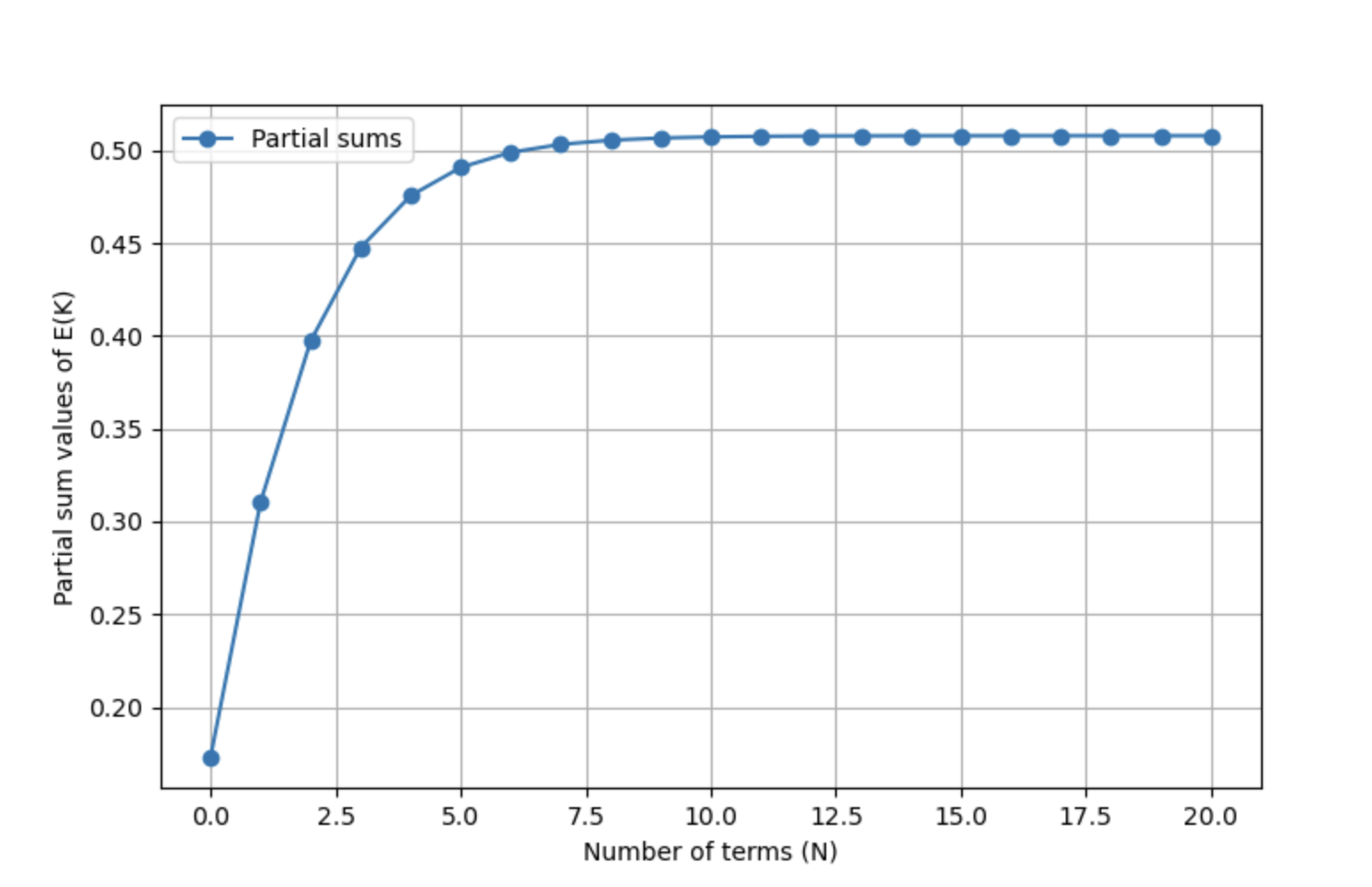}
    \caption{First 20 partial sums of $E(K_T)$ when $d=2$}
    \label{fig:placeholder}
\end{figure}

Using Theorem~\ref{thm:general-case}, the above calculation can be straightforwardly generalized to give the entropy $E(K_T)$ for any positive Toeplitz operator of the form
\[
T= c_0I +\sum_{|w|=N}c_w L^w +\overline{c_w}L^{w*},
\]
the interested reader may easily supply the details. 

\section{Random tridiagonal matrices}
It is remarkable that the limiting eigenvalue distribution $\mu_T$ that we have obtained in the case of $L_1+L_1^*$ coincides with the limiting eigenvalue distribution of a simple random tridiagonal matrix model. In this section we describe the random tridiagonal model and give a heuristic explanation for the coincidence. 

Random tridiagonal matrices arise naturally in a variety of contexts, including models of stochastic processes, hierarchical dynamics, and transfer operators with nearest-neighbor interactions; see  \cite{MR3436667} for connections of the particular model we will consider to mathematical physics. See \cite{popescu-2025} for a treatment of fairly general random tridiagonal models; we will import one relatively simple special case of the results from \cite{popescu-2025}.

Consider a selfadjoint tridiagonal $N\times N$ random matrix
\[ 
A_N = \begin{bmatrix}
0 & a_{1} & 0 &  0 & \cdots \\
a_{1} & 0 & a_{2} & 0 & \cdots \\
0 & a_{2} & 0 &  a_{3} & \cdots \\
0& 0 & a_{3} &  0 & \cdots \\
\vdots & \ddots &  \ddots & \ddots & \ddots
\end{bmatrix},
\]
where the entries $a_i$ are independent and $p$-Bernoulli distributed, i.e,

  \[ a_i=
    \begin{cases}
        1 & \text{with probability $p$} \\
        0 & \text{with probability $1-p$} \\
    \end{cases}
\]
\\

Define the empirical eigenvalue distribution of $A_N$ as:
\begin{align*}
    \rho_N = \frac{1}{N} \sum_{j=1}^N \delta_{\lambda_j(A_N)}
\end{align*}
where $\lambda_j(A_N)$ are the eigenvalues of $A_N$. For each $N$ the $\rho_N$ form an ensemble of random probability measures supported in $[-2,2]$. \\

We have the following theorem from \cite[Proposition 4.2]{popescu-2025}:
\begin{proposition}
The Bernoulli tridiagonal matrices $A_N$ have limiting eigenvalue distribution
  \begin{align*}
\mu=   (1-p)^2\sum_{h=0}^\infty  p^h \mu_h
  \end{align*}
(that is, the $\rho_N$ converge weak-* almost surely to $\mu$).   
\end{proposition}

These $\mu_h$ are the same as those appearing in Theorem~\ref{thm:special-case}. So, we see 
that the limiting eigenvalue distribution of the truncated  multi-Toeplitz operators
\[
   T_N=P_N(L_1+ L_1^*)P_N
\]
coincides with the limiting distribution of the Bernoulli tridiagonal model, when the probability is $p=\frac1d$. This agreement can be understood heuristically by the following argument; we work in the case $d=2$. 

Consider $T=L_1+L_1^*$ acting in the Fock space $\mathcal{F}_2$. The multi-Toeplitz matrix (\ref{eqn:big-toeplitz}) of $T$ has a sparse structure: as observed in the proof of Theorem~\ref{thm:special-case}, this matrix may be viewed as the incidence matrix of a graph whose nodes are the nodes of a rooted binary tree. Only the "left" edges between parent and child nodes of the tree are kept (see Figure~\ref{fig:tree}). This structure may be viewed as a deterministic percolation model on the binary tree. 

Because only the left edges	are kept, the graph decomposes into linear chains. After four generations, for example, one finds that one-eighth of the clusters have length $4$, one-eighth have length $3$,  one-fourth have length $2$, and one half have length $1$. In general, clusters of size $k$ occur with frequency asymptotically proportional to $2^{-k}$. The limiting graph therefore consists of linear components whose size distribution is geometric. (This is the intuition that underlies the proof of Theorem~\ref{thm:special-case}.)

Now compare this with the random tridiagonal model described above: the model produces a random percolation model on $\mathbb{N}$. In this setting as well, the (random) graph decomposes into linear clusters, and the probability of observing a chain of length $k$ is proportional to $2^{-k}$.

\begin{figure}[h]
\begin{center}

\begin{tikzpicture}[scale=1, every node/.style={circle, draw, fill=white!20, minimum size=5mm}]

\def\n{10}

\foreach \i in {1,...,\n} {
    \node (\i) at (\i,0) {\i};
}

\foreach \i/\j in {1/2, 3/4, 4/5, 6/7, 8/9, 9/10} {
    \draw [red, very thick](\i) -- (\j);
}

\end{tikzpicture}

\end{center}
\caption{Random percolation process determined by the random tridiagonal matrices $A_N$.}
\end{figure}
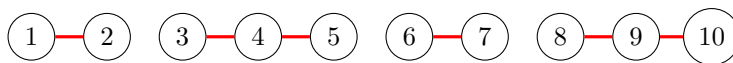

Thus, in both the deterministic multi-Toeplitz case and the random tridiagonal model, we get geometrically distributed linear clusters, so one would expect the limiting eigenvalue distributions to reflect these common statistics, as indeed we have seen. However, this viewpoint does not immediately extend to more complicated symbols or to operators involving multiple creation and annihilation terms, where the underlying graph structure is no longer composed of simple linear clusters and the combinatorics become substantially more intricate. We leave this problem for future investigation.

\printbibliography

\end{document}